\documentclass[10pt]{amsart}
\usepackage{amssymb,mathtools}
\usepackage{microtype}
\usepackage[hidelinks]{hyperref}

\theoremstyle{plain}
\newtheorem{theorem}{Theorem}[section]
\newtheorem{lemma}[theorem]{Lemma}
\newtheorem{corollary}[theorem]{Corollary}
\theoremstyle{remark}
\newtheorem{remark}[theorem]{Remark}
\numberwithin{equation}{section}

\newcommand{\C}{\mathbb C}
\newcommand{\E}{\mathbb E}
\newcommand{\Var}{\operatorname{Var}}
\newcommand{\cum}{\operatorname{cum}}
\newcommand{\diam}{\operatorname{diam}}
\newcommand{\Li}{\operatorname{Li}_2}
\newcommand{\partitions}{\mathcal P}

\begin{document}

\title[Higher-order sum rules for planar GEF zeros]
{Higher-order sum rules for zeros of the planar Gaussian entire function}

\author{Xiaolin Zhu}
\address{Taiyuan Normal University, Taiyuan, P. R. China}
\email{xiaolinzhu@tynu.edu.cn}
\thanks{X. Zhu is supported by the Tianyuan Fund for Mathematics of the National Natural Science Foundation of China (Grant No. 12526614).}

\author{Feng Guo}
\address{School of Mathematics, Nanjing University of Aeronautics and
Astronautics, Nanjing 211106, P. R. China}
\email{70207994@nuaa.edu.cn}
\thanks{Feng Guo is the corresponding author.}

\date{}
\hypersetup{pdftitle={Higher-order sum rules for zeros of the planar Gaussian entire function},pdfauthor={Xiaolin Zhu and Feng Guo}}

\begin{abstract}
For each fixed integer $k\ge2$, we prove that integrating one
coordinate of the truncated correlation of order $k$ for zeros
of the planar Gaussian entire function (GEF) gives $-(k-1)$ times
the truncated correlation of order $k-1$. The integral converges
absolutely, and the identity holds pointwise, including at
coincident fixed coordinates by continuous extension. This answers the marginal-identity part
of Krishnapur and Yogeshwaran's question for the planar GEF.
We also prove a general criterion for simple point processes:
under local moment and absolute integrability assumptions, the
identities hold almost everywhere if there are compactly supported
cutoffs between zero and one, eventually equal to one on each
compact set, whose linear statistics have variances tending to zero.
For GEF zeros, strong clustering and variance bounds for smooth
linear statistics verify these assumptions. The
proof combines a mixed-cumulant estimate with the partition formula
for ordinary and factorial cumulants.
\end{abstract}

\maketitle

\section{Introduction}\label{sec:intro}
For zeros of the planar Gaussian entire function (GEF),
Krishnapur and Yogeshwaran asked whether integrating one
coordinate of the truncated correlation of order $k$ gives
exactly $-(k-1)$ times the truncated correlation of order
$k-1$ \cite[Assumption~4.2 and Example~4.9]{KY}.

For a stationary point process of intensity $\lambda$ and truncated
pair density $\rho_2^T$, the order-two recursion requires the
hyperuniformity sum rule
\[
 \int\rho_2^T(0,w)\,dw=-\lambda.
\]
For planar GEF zeros, this identity follows from the pair-correlation
calculation of Forrester and Honner \cite[Section~4.1]{FH}.
Nazarov and Sodin subsequently proved strong clustering at every
fixed order \cite[Theorem~1.4]{NS}, yielding absolute integrability
of the higher truncated correlations \cite[Claims~4.1--4.2]{NS}.
Thus the higher marginal integrals were known to converge; the
question in \cite[Example~4.9]{KY} concerns their exact values for
integer orders $k\ge3$.

Theorem~\ref{thm:GEF} proves the one-coordinate marginal identity
for every fixed integer $k\ge2$. The integral converges absolutely, and the identity
holds pointwise, including at coincident fixed coordinates by
continuous extension.
The question in \cite[Example~4.9]{KY}, first posted in 2024,
also asks about a local identity for logarithmic cumulants and
more general random analytic functions.

We also prove a criterion for simple point processes
(Theorem~\ref{thm:abstract}), giving the marginal identities almost
everywhere when the stated local moment, absolute integrability
and cutoff assumptions hold. The proof of the criterion requires
neither stationarity nor a Gaussian structure. We apply it to GEF zeros.

The marginal identities express cancellations needed to estimate
higher cumulants of linear statistics. In the ordinary
cumulant expansion, lower-order truncated correlations arise from
tests evaluated at the same point. The marginal identities combine
these terms into a single integral against the highest-order
truncated correlation, with a test factor built from products of
differences \cite[Proposition~4.12]{KY}. This factor vanishes on the
full diagonal, while clustering controls the correlation away from
that diagonal. Theorem~\ref{thm:GEF} therefore verifies the marginal
hypothesis in the GEF application of \cite[Example~4.10]{KY},
subject to the remaining hypotheses of
\cite[Theorems~4.3 and~4.5]{KY}. Central limit theorems for smooth
linear statistics of GEF zeros were proved by
Sodin and Tsirelson \cite{ST}; Nazarov and Sodin \cite{NSsmooth}
extended them to broader classes of test functions under the
conditions stated there.

\subsection{Definitions and main results}
Throughout the paper \(dA=dx\,dy\) on \(\C\), and
\(\Delta=\partial_x^2+\partial_y^2\).
A centered complex Gaussian vector is called proper if all its
pseudo-covariances $\E\eta_i\eta_j$ vanish.
Let \(\xi_n\) be independent centered proper complex Gaussian variables,
each with density \(\pi^{-1}e^{-|u|^2}\). Define
\[
 F(z)=\sum_{n=0}^{\infty}\xi_n\frac{z^n}{\sqrt{n!}}.
\]
All finite vectors of values and derivatives of $F$ are centered and
jointly proper.
Write \(\mathcal Z\) for its zero counting measure.
For a locally finite point measure $X$ and a bounded measurable
compactly supported function $h$, write $X(h)=\int h\,dX$.
We say that a point process $X$ has all local moments if
$\E X(K)^p<\infty$ for every compact set $K$ and every integer $p\ge1$.
We define the symmetric factorial correlation density \(\rho_j\) by
\[
 \E\sum_{(u_1,\ldots,u_j)\in\mathcal Z^j_{\ne}}
 g(u_1,\ldots,u_j)
 =
 \int_{\C^j}g(z_1,\ldots,z_j)\rho_j(z_1,\ldots,z_j)
 \prod_{i=1}^j dA(z_i).
\]
Here $g$ is nonnegative, measurable and compactly supported, and the sum
is over ordered tuples of distinct zeros. Existence of the densities and
their continuous extensions is proved in Lemma~\ref{lem:local}.
For \(j\ge1\), set \([j]=\{1,\ldots,j\}\) and let \(\partitions_j\)
denote the set of partitions of this set. For \(\pi\in\partitions_j\),
write \(|\pi|\) for its number of blocks. The truncated correlation
functions are defined by
\begin{equation}\label{eq:factorial-cumulants}
 \rho_j^T(z_1,\ldots,z_j)
 =\sum_{\pi\in\partitions_j}(-1)^{|\pi|-1}(|\pi|-1)!
       \prod_{B\in\pi}\rho_{|B|}(z_B).
\end{equation}
Here \(z_B\) is the tuple of coordinates indexed by \(B\).
In particular \(\rho_1^T=\rho_1\). At the next two orders,
\begin{align*}
 \rho_2^T(z_1,z_2)
 &=\rho_2(z_1,z_2)-\rho_1(z_1)\rho_1(z_2),\\
 \rho_3^T(z_1,z_2,z_3)
 &=\rho_3(z_1,z_2,z_3)
   -\sum_{\{i,j\}\subset[3]}\rho_2(z_i,z_j)\rho_1(z_\ell)
   +2\prod_{i=1}^3\rho_1(z_i),
\end{align*}
where $\ell$ is the index outside $\{i,j\}$ in each summand.
These truncated correlations are the factorial cumulant densities. Ordinary
cumulant measures of order two or higher also have contributions
supported on diagonals.
\begin{samepage}
\begin{theorem}\label{thm:GEF}
Let \(k\ge2\) be a fixed integer. For every tuple
\((z_1,\ldots,z_{k-1})\in\C^{k-1}\), the continuous factorial
cumulant densities of the planar GEF satisfy the following
identity, with an absolutely convergent integral:
\begin{equation}\label{eq:GEF-sum-rule}
 \displaystyle
 \int_\C\rho_k^T(z_1,\ldots,z_{k-1},w)\,dA(w)
 =-(k-1)\rho_{k-1}^T(z_1,\ldots,z_{k-1}).
\end{equation}
At coincident coordinates, the densities are defined by the
continuous extensions in Lemma~\ref{lem:local}.
\end{theorem}
\end{samepage}

For $k=3$, the identity reads
$\int_\C\rho_3^T(z_1,z_2,w)\,dA(w)=-2\rho_2^T(z_1,z_2)$.
The constants in the proof may depend on the order $k$. Repeated
integration gives the iterated formula of \cite[(4.5)]{KY} for GEF zeros.
\begin{corollary}\label{cor:iterated}
Let \(k>r\ge1\) be integers, and fix
\((z_1,\ldots,z_r)\in\C^r\), allowing coincident coordinates.
Then the following integral converges absolutely and satisfies
\[
 \int_{\C^{k-r}}\rho_k^T(z_1,\ldots,z_k)
       \prod_{i=r+1}^k dA(z_i)
 =
 (-1)^{k-r}\frac{(k-1)!}{(r-1)!}
       \rho_r^T(z_1,\ldots,z_r).
\]
In particular the integral over all but one coordinate is
\((-1)^{k-1}(k-1)!/\pi\).
\end{corollary}

The proof uses a criterion for simple point processes. For a process
on $\mathbb R^d$, let $\rho_j$ denote its factorial correlation
densities with respect to Lebesgue measure, and define $\rho_j^T$
by \eqref{eq:factorial-cumulants}.

\begin{theorem}\label{thm:abstract}
Let \(X\) be a simple point process on \(\mathbb R^d\) with all local
moments and locally integrable factorial correlation densities.
Fix an integer \(k\ge2\), and write
$z=(z_1,\ldots,z_{k-1})$ and $dz=\prod_{i<k}dz_i$. Assume that
\[
 \int_{K^{k-1}\times\mathbb R^d}|\rho_k^T(z,w)|\,dz\,dw<\infty
 \quad\text{for every compact }K\subset\mathbb R^d.
\]
Suppose there is a sequence of measurable compactly supported
functions $\chi_n:\mathbb R^d\to[0,1]$ such that, for every compact
$K$, $\chi_n=1$ on $K$ for all sufficiently large $n$, and
$\Var X(\chi_n)\to0$.
Then
\begin{equation}\label{eq:abstract-sum-rule}
 \int_{\mathbb R^d}\rho_k^T(z_1,\ldots,z_{k-1},w)\,dw
 =-(k-1)\rho_{k-1}^T(z_1,\ldots,z_{k-1})
\end{equation}
for almost every \((z_1,\ldots,z_{k-1})\). Assume, in addition,
that the factorial cumulant densities have continuous representatives
on tuples of distinct points, and that for each compact
$K$ there is an integrable $M_K(w)$ such that
$|\rho_k^T(z,w)|\le M_K(w)$ whenever $z\in K^{k-1}$ has distinct
coordinates and $w$ differs from each of them. Then the identity holds
at every tuple of distinct points.
\end{theorem}

\subsection{Proof outline and related work}
To prove Theorem~\ref{thm:GEF}, we first show how a cutoff statistic
with vanishing variance forces the marginal identity for a general
point process. If the total number of points is finite and
deterministic, a mixed cumulant of order at least two vanishes when
one argument is that total number, and the partition formula gives
the marginal identities \cite[Example~4.6 and Appendix~B]{KY}.
For an infinite process, we use a compactly supported cutoff whose
centered statistic tends to zero in $L^2$.

Under the assumptions of Theorem~\ref{thm:abstract}, fix bounded
measurable compactly supported real tests $h_1,\ldots,h_{k-1}$.
After centering the cutoff argument, Cauchy--Schwarz in the moment
expansion gives
\[
 \bigl|\cum(X(h_1),\ldots,X(h_{k-1}),X(\chi_n))\bigr|
 \le C_{h,k}\sqrt{\Var X(\chi_n)}\longrightarrow0.
\]
Local moments control the products of fixed statistics, so
$C_{h,k}$ is independent of the growing cutoff support
(Lemma~\ref{lem:mixed-cumulant}).

Choose the fixed tests with pairwise disjoint supports. In
\eqref{eq:partition-expansion}, any block containing two fixed
tests contributes zero. Only the partition into singletons and
the $k-1$ partitions whose sole nonsingleton block is $\{j,k\}$
remain. For large $n$, $h_j\chi_n=h_j$, so symmetry makes these
$k-1$ terms equal, each with coefficient $+1$
(Lemma~\ref{lem:partition-expansion}).
Writing $z=(z_1,\ldots,z_{k-1})$ and $dz=\prod_{i<k}dz_i$, we obtain
\begin{align*}
 &\cum(X(h_1),\ldots,X(h_{k-1}),X(\chi_n))\\
 &\quad=\int\prod_{i<k}h_i(z_i)
   \left[\int\chi_n(w)\rho_k^T(z,w)\,dw
          +(k-1)\rho_{k-1}^T(z)\right]dz.
\end{align*}

The absolute integrability assumption and $0\le\chi_n\le1$ permit
dominated convergence in this identity. The bracket without the
cutoff therefore has integral zero against every product of tests
with pairwise disjoint supports. Such products determine a locally
integrable function on configurations of distinct points; the
remaining diagonals have Lebesgue measure zero. This proves the
almost-everywhere conclusion of Theorem~\ref{thm:abstract}.

For GEF zeros, the distributional identity
$\mathcal Z=(2\pi)^{-1}\Delta\log|F|$ reduces the variance of a smooth
linear statistic to the covariance of a logarithmic Gaussian field.
After normalization by $e^{-|z|^2/2}$, this covariance is an
integrable function of $z-w$. Young's inequality gives
$\Var\mathcal Z(\phi)\le C\|\Delta\phi\|_2^2$ for real
$\phi\in C_c^\infty(\C)$ (Lemmas~\ref{lem:log-covariance}
and~\ref{lem:cutoff}). Choose a smooth $0\le\chi\le1$ equal to one
on the unit disk and supported in the disk of radius two. Then
$\chi_R(z)=\chi(z/R)$ is eventually one on every compact set, and
\[
 \Var\mathcal Z(\chi_R)
 \le C\|\Delta\chi_R\|_2^2
 =CR^{-2}\|\Delta\chi\|_2^2\longrightarrow0.
\]
The corresponding mixed cumulant is $O(R^{-1})$ for the fixed tests.

We derive absolute integrability of the truncated correlations
in one coordinate from the strong clustering estimates of
Nazarov and Sodin \cite[Theorems~1.3--1.4]{NS}.
In moment--cumulant inversion, partitions confined to two
separated groups cancel. Each remaining lower-order term contains
a block meeting both groups; induction gives decay for that block,
while the other factors are globally bounded.
A configuration of diameter $D$ can be split into groups separated
by at least $D/(j-1)$, giving Lemma~\ref{lem:decay}:
\[
 |\rho_j^T(Z)|\le A_j e^{-a_j\diam(Z)^2},\qquad
 a_j=\frac1{8(j-1)^2},\quad j\ge2.
\]
Since $\diam(z,w)\ge|w-z_1|$, this bound gives a common integrable
majorant in $w$ when the fixed coordinates range over a compact set.

Lemma~\ref{lem:local} supplies local moments and continuity of
the correlation densities.
Together with the variance bound and common majorant, this verifies
the hypotheses of Theorem~\ref{thm:abstract}. Dominated convergence
also makes the marginal continuous at collisions, so its
almost-everywhere identity extends to every fixed tuple, proving
Theorem~\ref{thm:GEF}.

Small-variance extensions of local tests also appear in the rigidity
criterion of Ghosh and Peres \cite[Theorem~6.1]{GP}.
Here absolute integrability allows us to identify higher-order
marginals.
The argument in \cite[Appendix~C]{KY} instead studies local boundary
terms arising from Laplacians of logarithmic cumulants.

We prove the general criterion in Section~\ref{sec:criterion},
verify its hypotheses for GEF zeros in Sections~\ref{sec:local}
and~\ref{sec:cumulants}, and obtain the marginal identities in
Section~\ref{sec:hierarchy}. Appendix~\ref{app:log-covariance}
proves the logarithmic covariance formula used in
Section~\ref{sec:cumulants}.

\section{A cumulant criterion for marginal sum rules}\label{sec:criterion}
To prove Theorem~\ref{thm:abstract}, we use the partition formula
for mixed cumulants of linear statistics in terms of factorial
cumulant densities \cite[(4.15) and Appendix~A]{KY}.

For real random variables $Y_1,\ldots,Y_j$ for which the expectations
below are finite, define
\begin{equation}\label{eq:ordinary-cumulants}
 \cum(Y_1,\ldots,Y_j)
 =\sum_{\sigma\in\partitions_j}(-1)^{|\sigma|-1}(|\sigma|-1)!
       \prod_{B\in\sigma}\E\prod_{i\in B}Y_i.
\end{equation}
\begin{lemma}\label{lem:partition-expansion}
Let $X$ be a simple point process on $\mathbb R^d$ with all local
moments and factorial correlation densities. Let
$h_1,\ldots,h_k$ be bounded measurable compactly supported real functions.
Then
\begin{equation}\label{eq:partition-expansion}
 \cum(X(h_1),\ldots,X(h_k))
 =
 \sum_{\pi\in\partitions_k}
 \int\prod_{B\in\pi}\prod_{i\in B}h_i(x_B)\,
       \rho_{|\pi|}^T((x_B)_{B\in\pi})\prod_{B\in\pi}dx_B.
\end{equation}
There is one integration variable $x_B$ for each block of $\pi$.
\end{lemma}
\begin{proof}
Expand a product of linear statistics by grouping indices
evaluated at the same point into blocks of a partition $\pi$.
Distinct blocks correspond to distinct points, so the expected
contribution is an integral against $\rho_{|\pi|}$.

Apply this expansion to each block moment in
\eqref{eq:ordinary-cumulants}. For fixed $\pi$, the contributing
outer partitions are the coarsenings $\sigma\ge\pi$, or
equivalently the partitions of the blocks of $\pi$.
A coarsening with $r$ blocks has coefficient
$(-1)^{r-1}(r-1)!$ and groups the variables $x_B$ accordingly.
Summing over these coarsenings gives
\eqref{eq:factorial-cumulants} at order $|\pi|$, proving
\eqref{eq:partition-expansion}.

All tests are supported in one compact set $K$, and $\E X(K)^k<\infty$
controls the absolute moments. Diagonals between variables in
different moment factors have Lebesgue measure zero. Hence each
partition has coefficient $+1$ in \eqref{eq:partition-expansion};
the alternating coefficients are contained in $\rho_{|\pi|}^T$.
\end{proof}

\begin{lemma}\label{lem:mixed-cumulant}
Let $X$ be a point process with all local moments. Fix an integer $k\ge2$
and bounded measurable compactly supported real functions $h_1,\ldots,h_{k-1}$.
Then, for every bounded measurable compactly supported real function $v$,
\[
 |\cum(X(h_1),\ldots,X(h_{k-1}),X(v))|
 \le C_{h,k}\sqrt{\Var X(v)}.
\]
The constant may depend on the process, the fixed tests and $k$,
but not on $v$ or its support.
\end{lemma}
\begin{proof}
A cumulant of order at least two is unchanged by centering its
last argument: for a partition of the first $k-1$ indices into
$r$ blocks, the terms with a constant last argument have total
coefficient $(-1)^r r!+r(-1)^{r-1}(r-1)!=0$.

Put $Y=X(v)-\E X(v)$. In each term of
\eqref{eq:ordinary-cumulants}, the block containing $Y$ contributes
$\E[Y\prod_{i\in S}X(h_i)]$. This vanishes for $S=\varnothing$;
otherwise Cauchy--Schwarz gives
\[
 \left|\E\left[Y\prod_{i\in S}X(h_i)\right]\right|
 \le\sqrt{\Var X(v)}
 \left(\E\left|\prod_{i\in S}X(h_i)\right|^2\right)^{1/2}.
\]
Choose a compact $K$ containing the fixed supports. The squared
product is bounded by a constant times $X(K)^{2(k-1)}$, and all
other moment factors involve only the fixed statistics.
Summing over partitions proves the estimate, with a constant
independent of $v$ and its support. Local moments through order
$2(k-1)$ suffice.
\end{proof}

\begin{proof}[Proof of Theorem~\ref{thm:abstract}]
Fix bounded compactly supported real measurable functions
$h_1,\ldots,h_{k-1}$ with pairwise disjoint supports, all contained
in a compact set $K$. Write $z=(z_1,\ldots,z_{k-1})$ and
$dz=\prod_{i<k}dz_i$. In \eqref{eq:partition-expansion}, any block
containing two fixed indices contributes zero. The surviving
partitions are the partition into singletons and the $k-1$
partitions whose only nonsingleton block is $\{j,k\}$.
For large $n$, $\chi_n=1$ on $K$, so $h_j\chi_n=h_j$.
The $k-1$ lower-order integrals are equal by symmetry, each with
coefficient $+1$, giving
\begin{align*}
 &\cum(X(h_1),\ldots,X(h_{k-1}),X(\chi_n))\\
 &\quad=\int\prod_{i<k}h_i(z_i)\chi_n(w)\rho_k^T(z,w)\,dz\,dw\\
 &\qquad +(k-1)\int\prod_{i<k}h_i(z_i)\rho_{k-1}^T(z)\,dz.
\end{align*}

The left-hand side tends to zero by Lemma~\ref{lem:mixed-cumulant}.
The first integrand is bounded independently of $n$ by the
integrable function
\[
 \left(\prod_{i<k}\|h_i\|_\infty\right)
 \boldsymbol 1_{K^{k-1}}(z)|\rho_k^T(z,w)|.
\]
Dominated convergence removes the cutoff. The second integral
is finite by local integrability of $\rho_{k-1}^T$.
Thus the locally integrable function
\[
 H(z)=\int\rho_k^T(z,w)\,dw+(k-1)\rho_{k-1}^T(z),
\]
defined almost everywhere, has integral zero against every such
product of tests.

Let $U_1,\ldots,U_{k-1}$ be pairwise disjoint relatively compact
open sets. Taking the $h_i$ to be indicators of measurable
subsets compactly contained in $U_i$ shows that the signed measure
with density $H$ vanishes on their products. Uniqueness of finite
signed measures on rectangles, applied on compact subboxes and
followed by exhaustion, gives $H=0$ almost everywhere on
$\prod_iU_i$. A countable basis of relatively compact balls gives
a countable cover of the distinct configurations by such
products. The remaining diagonals have Lebesgue measure zero,
proving \eqref{eq:abstract-sum-rule}. For $k=2$, disjointness is
vacuous and the same argument applies.

Under the additional assumptions, let distinct tuples $z^{(m)}$
converge to a distinct tuple $z$. For $w$ outside the coordinates
of $z$, the tuple $(z^{(m)},w)$ is eventually distinct, so continuity gives
$\rho_k^T(z^{(m)},w)\to\rho_k^T(z,w)$.
All fixed coordinates lie in one compact set, whose integrable
majorant permits dominated convergence. Hence the marginal
integral is continuous on the distinct configurations.
Together with continuity of $\rho_{k-1}^T$, this makes $H$
continuous there. Its almost-everywhere vanishing gives $H=0$
everywhere on this set.
\end{proof}

\section{Correlation estimates for planar GEF zeros}\label{sec:local}
We establish the local moments and correlation bounds needed to
apply Theorem~\ref{thm:abstract} in Section~\ref{sec:hierarchy}.
Local repulsion gives continuity at coincident coordinates, while
clustering gives an integrable bound in the last coordinate,
uniform when the remaining coordinates vary in a compact set.
\begin{lemma}\label{lem:local}
Almost surely, the series \(F\) defines an entire function which is
not identically zero and has only simple zeros. For every compact
set, the number of zeros has finite moments of all positive orders.
The ordered factorial correlation densities exist and are smooth at
tuples of distinct points. For every integer \(j\ge1\) and every
\(M>0\), there is \(C_{j,M}<\infty\) such that
\begin{equation}\label{eq:local-repulsion}
 0\le\rho_j(z_1,\ldots,z_j)
 \le C_{j,M}\prod_{a<b}|z_a-z_b|^2,\qquad |z_i|\le M.
\end{equation}
For \(j\ge2\), setting \(\rho_j=0\) whenever two coordinates coincide
gives a continuous extension to \(\C^j\). The partition formula
\eqref{eq:factorial-cumulants} then extends each \(\rho_j^T\)
continuously to \(\C^j\).
\end{lemma}
\begin{proof}
For every $M>0$,
$\E\sum_n|\xi_n|M^n/\sqrt{n!}<\infty$, so the series is almost
surely uniformly convergent on the disk of radius $M$. Taking the
intersection of these events over positive integer radii shows that $F$ is almost
surely entire. Since $\xi_0\ne0$ almost surely, $F$ is not identically
zero.
Condition on all coefficients except $\xi_0$ and write $F=\xi_0+G$.
Almost surely $G'$ is not identically zero, so its zeros form a
countable set. A multiple zero of $F$ would require $\xi_0=-G(z)$ at
one of these points. Since $\xi_0$ is independent of $G$ and has a
non-atomic distribution, this event has conditional probability zero.

For $0<M<R$, Jensen's formula gives
\[
 N_M(F)\log(R/M)
 \le \log^+\sup_{|z|\le R}|F(z)|+\log^-|\xi_0|,
\]
where $N_M(F)$ counts zeros in the disk of radius $M$, and
$\log^+t=\max(\log t,0)$, $\log^-t=\max(-\log t,0)$.
The inequality follows first for radii with no zeros on the boundary
and then for arbitrary radii by passage to the limit. For every integer $p\ge1$,
Minkowski's inequality and the Gaussian density give
\begin{align*}
 \left\|\sup_{|z|\le R}|F(z)|\right\|_p
 &\le\sum_{n\ge0}\|\xi_n\|_p R^n/\sqrt{n!}<\infty,\\
 \E(\log^-|\xi_0|)^p
 &\le2\int_0^1r(-\log r)^p\,dr<\infty.
\end{align*}
Here $\|\cdot\|_p$ is the $L^p$ norm on the probability space.
Since $\log^+t\le t$, Jensen's bound gives all positive integer
moments of $N_M(F)$, hence all positive moments of the zero count
on any compact set.

The local repulsion bound follows from \cite[Theorem~1.3]{NS}.
For existence and smoothness of the densities, we use the complex
Kac--Rice formula \cite[Corollary~3.4.2]{HKPV}.
The Gaussian vector of values and derivatives at distinct points is
nondegenerate: if a linear combination has zero variance, the
corresponding linear functional annihilates all monomials. Hermite
interpolation then forces all coefficients of the combination to vanish.
For $Z=(z_1,\ldots,z_j)$, the covariance blocks and Kac--Rice
density are
\[
 \begin{gathered}
 A_{ab}=\E F(z_a)\overline{F(z_b)},\quad
 B_{ab}=\E F'(z_a)\overline{F(z_b)},\quad
 C_{ab}=\E F'(z_a)\overline{F'(z_b)},\\
 \rho_j(Z)=\frac{\operatorname{per}(C-BA^{-1}B^*)}{\pi^j\det A}.
 \end{gathered}
\]
Here $\operatorname{per}D=\sum_{\tau\in S_j}\prod_{i=1}^j D_{i,\tau(i)}$
is the permanent and $S_j$ is the permutation group. Joint properness
implies that the conditional derivative vector, given zero values
at the fixed points, is proper with covariance $C-BA^{-1}B^*$.
The covariance matrices depend smoothly on the points, and
$\det A>0$ when the points are distinct. Hence $\rho_j$ is smooth
away from collisions. The right-hand side of \eqref{eq:local-repulsion} tends
to zero as two coordinates approach each other within a fixed compact
set. Thus extending each $\rho_j$, $j\ge2$, by zero at coincident
coordinates gives a continuous function. Since
$\rho_1=1/\pi$ by \eqref{eq:mean}, continuity of $\rho_j^T$ follows
from \eqref{eq:factorial-cumulants}.
\end{proof}

For $Z=(z_1,\ldots,z_j)$, write
$\diam(Z)=\max_{a,b}|z_a-z_b|$. For disjoint nonempty index sets $I,J$,
set $d(Z_I,Z_J)=\min_{a\in I,b\in J}|z_a-z_b|$.

Theorems 1.3 and 1.4 of Nazarov and Sodin \cite{NS} give, for every
integer \(j\ge1\),
\begin{equation}\label{eq:NS-bounded}
 0\le \rho_j(Z)\le B_j
       \prod_{a<b}\min\{|z_a-z_b|^2,1\},
\end{equation}
and, for every nonempty bipartition \([j]=I\sqcup J\),
\begin{equation}\label{eq:NS-clustering}
 |\rho_j(Z)-\rho_{|I|}(Z_I)\rho_{|J|}(Z_J)|
 \le C_j e^{-(d-\delta_j)^2/2},
 \qquad d=d(Z_I,Z_J)\ge2\delta_j.
\end{equation}
The bounds extend from distinct to coincident coordinates by
continuity; the separation condition in \eqref{eq:NS-clustering}
remains unchanged. All constants here depend only on \(j\).

\begin{lemma}\label{lem:decay}
For every integer \(j\ge2\), there is \(A_j<\infty\) such that
\begin{equation}\label{eq:diameter-bound}
 |\rho_j^T(Z)|
 \le A_j\exp\left[-\frac{\diam(Z)^2}{8(j-1)^2}\right].
\end{equation}
Put \(a_j=1/[8(j-1)^2]\). For every fixed tuple
\((z_1,\ldots,z_{j-1})\), the function
\(w\mapsto\rho_j^T(z_1,\ldots,z_{j-1},w)\) is absolutely integrable,
with integral of its absolute value at most \(\pi A_j/a_j\).
When the fixed coordinates range over a compact set, choose $M$
so that $|z_1|\le M$. A common integrable majorant in $w$ is
\begin{equation}\label{eq:compact-majorant}
 A_j e^{-a_j(|w|-M)_+^2}.
\end{equation}
Here $r_+=\max\{r,0\}$. The integral of \(\rho_j^T\) in the last coordinate is a continuous
function of the remaining coordinates.
\end{lemma}
\begin{proof}
The finite partition formula and \eqref{eq:NS-bounded} show that
$|\rho_j^T|$ is globally bounded by $T_j$, where we may take
\[
 T_j=\sum_{\pi\in\partitions_j}(|\pi|-1)!
               \prod_{B\in\pi}B_{|B|}.
\]
Combining \eqref{eq:NS-clustering} with the global bounds gives, for
every nonempty bipartition \([j]=I\sqcup J\) and every separation
\(d=d(Z_I,Z_J)\ge0\),
\[
 |\rho_j(Z)-\rho_{|I|}(Z_I)\rho_{|J|}(Z_J)|
 \le Q_j e^{-d^2/8}.
\]
For \(d\ge2\delta_j\), use \(d-\delta_j\ge d/2\); for smaller \(d\),
use the global bounds for the two terms.

We prove by induction on \(j\) that, for every nonempty bipartition
\([j]=I\sqcup J\),
\[
 |\rho_j^T(Z)|\le A_j e^{-d(Z_I,Z_J)^2/8}.
\]
The case \(j=2\) is the preceding estimate. The finite
moment--cumulant inversion of \eqref{eq:factorial-cumulants} is
\(\rho_j=\sum_{\pi\in\partitions_j}\prod_{B\in\pi}\rho_{|B|}^T\).
A partition refines \(I\sqcup J\) if each block lies entirely in $I$
or entirely in $J$. The sum over all such partitions is
\(\rho_{|I|}(Z_I)\rho_{|J|}(Z_J)\). Subtracting this product from
\(\rho_j(Z)\) leaves precisely the partitions with a block meeting
both $I$ and $J$. Separating the one-block partition gives
\[
 \rho_j^T(Z)=\rho_j(Z)-\rho_{|I|}(Z_I)\rho_{|J|}(Z_J)
 -\sum_{\substack{\pi\text{ mixes }I\sqcup J\\|\pi|\ge2}}
     \prod_{B\in\pi}\rho_{|B|}^T(Z_B).
\]
Every partition in the sum has a block of size between $2$ and
$j-1$ meeting both groups. Its restricted bipartition has separation
at least $d(Z_I,Z_J)$, so induction bounds its truncated correlation
by a constant times $e^{-d^2/8}$. The other factors are globally
bounded. Since $j$ is fixed, the finitely many partitions and
bipartitions admit a common constant $A_j$.

Let \(D=\diam(Z)>0\). Join two points when their distance is
strictly less than \(D/(j-1)\). If the resulting graph were connected,
a simple path with at most \(j-1\) edges would join two points at
distance \(D\), with total path length strictly less than \(D\), a
contradiction. Thus the graph is disconnected. A connected component
and its complement have separation at least \(D/(j-1)\). This proves
\eqref{eq:diameter-bound}; the case \(D=0\) follows from the global bound.

Since $D\ge|w-z_1|$, integrating
$A_j e^{-a_j|w-z_1|^2}$ gives the bound $\pi A_j/a_j$.
For $|z_1|\le M$, the inequality
$|w-z_1|\ge(|w|-M)_+$ gives the integrable majorant
\eqref{eq:compact-majorant}. The same majorant bounds the integrands
along any convergent sequence of fixed tuples in a common compact set.
Lemma~\ref{lem:local} and dominated convergence therefore give
continuity of the marginal, including at coincident coordinates.
\end{proof}

\section{Smooth statistics with vanishing variance}\label{sec:cumulants}
We obtain the smooth cutoffs for Theorem~\ref{thm:abstract} from
the Gaussian covariance formula below \cite[Lemma~2.2]{NSsmooth}.
Appendix~\ref{app:log-covariance} gives a proof covering the full
range $0\le q\le1$.
\begin{lemma}\label{lem:log-covariance}
Let $\left(\eta_1,\eta_2\right)$ be a centered proper complex
Gaussian vector with $\E|\eta_1|^2=\E|\eta_2|^2=1$.
Put $q=|\E\eta_1\overline{\eta_2}|^2$. Then
\[
 \operatorname{Cov}(\log|\eta_1|,\log|\eta_2|)
 =\tfrac14\Li(q),\qquad
 \Li(q)=\sum_{m\ge1}\frac{q^m}{m^2},\quad 0\le q\le1.
\]
\end{lemma}

\begin{lemma}\label{lem:cutoff}
Let \(\phi\in C_c^\infty(\C)\) be real. Then
\begin{align}
 \E\mathcal Z(\phi)&=\pi^{-1}\int_\C\phi\,dA,\label{eq:mean}\\
 \Var\mathcal Z(\phi)
 &=\frac1{16\pi^2}\iint_{\C^2}
  \Delta\phi(z)\Delta\phi(w)\Li(e^{-|z-w|^2})\,dA(z)dA(w)
 \le\frac{\zeta(3)}{16\pi}\|\Delta\phi\|_2^2.
 \label{eq:variance}
\end{align}
Here \(\zeta\) denotes the Riemann zeta function, and the $L^2$ norm
is taken with respect to $dA$.
Fix a real $\chi\in C_c^\infty(\C)$ with $0\le\chi\le1$, equal to
one on the unit disk and supported in the disk of radius two.
For $R>0$, put $\chi_R(z)=\chi(z/R)$. As $R\to\infty$, these
functions are eventually equal to one on every fixed compact set, and
\[
 \Var\mathcal Z(\chi_R)
 \le \frac{\zeta(3)}{16\pi R^2}\|\Delta\chi\|_2^2\longrightarrow0.
\]
\end{lemma}
\begin{proof}
Local factorization of a nonzero analytic function and
\(\Delta\log|z|=2\pi\delta_0\) give
\(\mathcal Z=(2\pi)^{-1}\Delta\log|F|\) distributionally.
The covariance kernel of $F$ is \(e^{z\bar w}\), so
$\E|F(z)|^2=e^{|z|^2}$. Thus
\(\widetilde F(z)=e^{-|z|^2/2}F(z)\) has standard proper complex
Gaussian marginals, and the squared modulus of its correlation at
$z,w$ is \(e^{-|z-w|^2}\).
Let $\eta$ be a standard centered proper complex Gaussian variable,
and put \(L(z)=\log|\widetilde F(z)|-\E\log|\eta|\).
Almost surely, $\log|\widetilde F|$ is locally integrable.
Moreover, $\E|L(z)|^2$ is finite and independent of $z$.
We may therefore integrate $L$ against $\Delta\phi$
both distributionally and in $L^2$.
Since $\Delta(|z|^2/2)=2$ and $\int\Delta\phi=0$, we obtain
\[
 \mathcal Z(\phi)-\pi^{-1}\int\phi
 =\frac1{2\pi}\int_\C L(z)\Delta\phi(z)\,dA(z).
\]
Taking expectations gives \eqref{eq:mean}. Applying Fubini's theorem
and Lemma~\ref{lem:log-covariance} to the second moment gives the
equality in \eqref{eq:variance}.
Set $Q(u)=\Li(e^{-|u|^2})$ and $g=\Delta\phi$.
The double integral in \eqref{eq:variance} is
$\int g(Q*g)\,dA$. Tonelli's theorem, Cauchy--Schwarz and
Young's inequality give
\begin{align*}
 \|Q\|_1
 &=\sum_{m\ge1}\frac1{m^2}
     \int_\C e^{-m|u|^2}\,dA(u)=\pi\zeta(3),\\
 \left|\int g(Q*g)\,dA\right|
 &\le\|g\|_2\|Q*g\|_2
 \le\pi\zeta(3)\|g\|_2^2.
\end{align*}
This proves the variance bound, while
$\|\Delta\chi_R\|_2^2=R^{-2}\|\Delta\chi\|_2^2$
gives the cutoff estimate.
\end{proof}
\section{The marginal identities for GEF zeros}\label{sec:hierarchy}
\begin{proof}[Proof of Theorem~\ref{thm:GEF}]
By Lemmas~\ref{lem:local} and~\ref{lem:decay}, the GEF zero
process is simple, has all local moments, and has continuous
factorial correlation densities. For every compact $K$, the common
integrable majorant \eqref{eq:compact-majorant} and the finite
volume of $K^{k-1}$ give the absolute integrability required in
Theorem~\ref{thm:abstract}. Applying the theorem with the cutoffs
from Lemma~\ref{lem:cutoff} gives \eqref{eq:GEF-sum-rule}
at every distinct fixed tuple.

Lemma~\ref{lem:decay} gives continuity of the marginal integral
also at collisions, and Lemma~\ref{lem:local} gives continuity of
$\rho_{k-1}^T$. Taking limits along distinct tuples proves the
identity at every fixed tuple.
\end{proof}

\begin{proof}[Proof of Corollary~\ref{cor:iterated}]
Write $Z=(z_1,\ldots,z_k)$ and $a_k=1/[8(k-1)^2]$.
The diameter comparison and Lemma~\ref{lem:decay} give
\begin{align*}
 \diam(Z)^2&\ge\frac1{k-1}\sum_{i=2}^k|z_i-z_1|^2,\\
 |\rho_k^T(Z)|&\le A_k
     \prod_{i=2}^k\exp\left[-\frac{a_k}{k-1}|z_i-z_1|^2\right].
\end{align*}
For fixed $z_1,\ldots,z_r$, the Gaussian factors in the remaining
coordinates are jointly integrable. The multiple integral is
therefore absolutely convergent. Fubini's theorem and successive
applications of \eqref{eq:GEF-sum-rule} give the factors
$-(k-1),-(k-2),\ldots,-r$, whose product is
$(-1)^{k-r}(k-1)!/(r-1)!$. The final assertion follows from
$\rho_1^T=\rho_1=1/\pi$, by \eqref{eq:mean}.
\end{proof}

\begin{remark}
At order two, Theorem~\ref{thm:GEF} recovers the known normalization
\[
 \int_\C\rho_2^T(z,w)\,dA(w)=-1/\pi.
\]
The cutoffs used in the proof vary smoothly across an annulus whose
width is comparable to its radius. The variance bound
\eqref{eq:variance} applies to these smooth functions; it does not
apply directly to sharp disk indicators.
\end{remark}

\appendix
\section{The Gaussian logarithmic covariance formula}\label{app:log-covariance}
\begin{proof}[Proof of Lemma~\ref{lem:log-covariance}]
Put $X=|\eta_1|^2$ and $Y=|\eta_2|^2$. For $q<1$, integration
against the proper complex Gaussian density gives, for $s,t\ge0$,
\[
 \E e^{-sX-tY}=[(1+s)(1+t)-qst]^{-1}.
\]
For $q=1$, equality in Cauchy--Schwarz makes the Gaussian
variables equal up to a constant phase, so $X=Y$ almost surely.
Since $X$ is exponential with mean one,
$\E e^{-(s+t)X}=(1+s+t)^{-1}$, giving the same formula.

For $x>0$ and $0<a<b<\infty$, put
$L_{a,b}(x)=\int_a^b(e^{-s}-e^{-sx})\,ds/s$.
Frullani's formula gives $L_{a,b}(x)\to\log x$ as
$a\downarrow0$ and $b\uparrow\infty$ through nested intervals.
For each fixed $x$, the integrand has constant sign as $s$ varies,
so $|L_{a,b}(x)|\le|\log x|$.
Exponential marginals have square-integrable logarithms; hence
$L_{a,b}(X)\to\log X$ and $L_{a,b}(Y)\to\log Y$ in $L^2$.

For finite $a,b$, Fubini expresses
$\operatorname{Cov}(L_{a,b}(X),L_{a,b}(Y))$ as the integral over
$[a,b]^2$ of
\[
 \frac1{st}\left(\frac1{P-qst}-\frac1P\right),
 \qquad P=(1+s)(1+t).
\]
This kernel is nonnegative. Expanding it using $0\le qst/P<1$
and letting the intervals increase to $(0,\infty)$, Tonelli and
the $L^2$ convergence give
\[
 \operatorname{Cov}(\log X,\log Y)
 =\sum_{m\ge1}q^m
   \left(\int_0^\infty\frac{s^{m-1}}{(1+s)^{m+1}}\,ds\right)^2
 =\sum_{m\ge1}\frac{q^m}{m^2}.
\]
The substitution $u=s/(1+s)$ evaluates the integral as $1/m$.
Since $\log|\eta_1|=\tfrac12\log X$ and
$\log|\eta_2|=\tfrac12\log Y$, the required covariance is one
quarter of this expression, for all $0\le q\le1$.
\end{proof}
\section*{Statements and Declarations}

\medskip
\noindent\textbf{AI Statement.}
The authors used artificial-intelligence tools for language editing,
\LaTeX\ formatting, and limited assistance with local mathematical
reasoning.  The proof strategy and mathematical development are the
authors' own; all mathematical content was independently written and checked by the
authors, who take full responsibility for it.

\end{document}